\documentclass[11pt, a4paper, reqno]{amsart}
\usepackage[margin=1.2in]{geometry}

\usepackage{amssymb}
\usepackage[noadjust]{cite}
\usepackage[protrusion=true,expansion=false]{microtype}
\usepackage[colorlinks=true,linkcolor=blue]{hyperref}
\usepackage{enumerate}
\usepackage{fouriernc}

\numberwithin{equation}{section}

\theoremstyle{plain}
\newtheorem{theorem}{Theorem}[section]
\newtheorem{proposition}[theorem]{Proposition}
\newtheorem{lemma}[theorem]{Lemma}
\newtheorem{corollary}[theorem]{Corollary}

\theoremstyle{definition}

\theoremstyle{remark}
\newtheorem{remark}[theorem]{Remark}

\DeclareMathOperator{\Sph}{Sph}
\newcommand{\ran}{\operatorname{r}}
\newcommand{\C}{\mathbb{C}}
\newcommand{\one}{\mathbf{1}}

\renewcommand{\emptyset}{\varnothing}

\title[Metric identification of projections]{The norm reveals projections in the positive unit sphere of an arbitrary C$^*$-algebra}

\author[A.M. Peralta and P. Saavedra]{Antonio M. Peralta \and Pedro Saavedra}
\address[A. M. Peralta]{Instituto de Matem{\'a}ticas de la Universidad de
Granada (IMAG), Departamento de An{\'a}lisis Matem{\'a}tico, Facultad de
Ciencias, Universidad de Granada, 18071 Granada, Spain.}
\email{aperalta@ugr.es}
\address[Pedro Saavedra]{Departamento de An\'{a}lisis Matem\'{a}tico,
Facultad de Ciencias, Universidad de Granada, 18071 Granada, Spain.}
\email{psaavedraortiz@ugr.es}
\date{October 26, 2026}

\begin{document}

\begin{abstract}
We show how a simple metric condition in terms of the C$^*$-norm of an arbitrary C$^*$-algebra $A$ can be employed to determine those norm-one positive elements which are projections. Concretely, if we write $S_{A^+}$ for the intersection of the positive cone and the unit sphere of $A$, we establish that an element $a \in S_{A^+}$ is a projection if, and only if, $$\{a\} = \left\{x\in S_{A^+}: \|x-z\| = 1 \hbox{ for all } z\in S_{A^+}\hbox{ with } \|a-z\|=1 \right\}.$$ This proves a conjecture which has been previously confirmed for von Neumann algebras and unital C$^*$-algebras. In the setting of not necessarily unital JB$^*$-algebras, the Shirshov--Cohn theorem is invoked to obtain a similar metric characterisation of projections. In the wider setting of JB$^*$-triples, we show how a metric characterisation in terms of inner ideals can be established to determine non-zero tripotents.
\end{abstract}
\maketitle

\section{Introduction}

Let $X$ be a real or complex normed space, and let $P,M\subseteq X$. The \emph{unit sphere around $M$ in $P$} is the set
\begin{equation}\label{eq:unit-sphere-around-a-set}
\Sph_P(M)
=\big\{x\in P:\|x-y\|=1\text{ for every }y\in M\big\}.
\end{equation}
We abbreviate $\Sph_P(\{a\})$ to $\Sph_P(a)$ and adopt the convention that 
$\Sph_P(\emptyset)=P$. For $a\in P$, the \emph{double sphere around $a$ in $P$} ($\Sph_P(\Sph_P(a))$ in short) consists of the points of $P$ at distance one
from every element of $\Sph_P(a)$. By symmetry of the distance, the double sphere around $a$ in $P$ always contains the element $a$. 

In the setting of C$^*$-algebras, the double sphere construction around an element in the positive sphere is employed to characterise projections in terms of the norm. Let $A$ be a C$^*$-algebra (or, more generally, a JB$^*$-algebra) whose closed positive cone is denoted by $A^+$. The symbol $S_{A^+}$ will stand for the positive sphere of $A$, that is, $$S_{A^+}=\big\{a\in A^+:\|a\|=1\big\}.$$ It was established in \cite[Proposition 2.2]{PeraltaAOT2018} that for each element $a\in S_{A^+}$, the condition $\Sph_{S_{A^+}}\bigl(\Sph_{S_{A^+}}(a)\bigr)=\{a\}$ suffices to conclude that $a$ is a projection. It was conjectured in the just quoted reference that the equivalence $$
{a \text{ is a non-zero projection}}
\quad\Longleftrightarrow\quad
{\Sph_{S_{A^+}}\bigl(\Sph_{S_{A^+}}(a)\bigr)=\{a\}}
$$ holds for all $a\in S_{A^+}$. Positive confirmations of this conjecture have been obtained in the following cases: $A= M_n(\mathbb{C})$ \cite[proof of Corollary on page 246]{Nagy2018}, $A=B(H)$ \cite[Corollary 2.4]{PeraltaAOT2018}, and when $A$ is a von Neumann algebra or an AW$^*$-algebra \cite[Lemma 4.1]{LeungNgWong2025}. In the recent paper \cite{PeraltaSaavedraUnital} we gave a complete proof of this conjecture when $A$ is a unital C$^*$-algebra or a unital JB$^*$-algebra. The conclusion for unital C$^*$-algebras was rediscovered in \cite[Corollary 3.6]{LeungNgWong2026}.

In this note, we provide a complete proof of the above conjecture for arbitrary (not necessarily unital) C$^*$- and JB$^*$-algebras (see Theorems~\ref{thm:cstar-projections} and~\ref{thm:jb-projections}). 

We have made a considerable effort to provide a complete and self-contained argument, organising the proof along one of the most direct possible routes. The approach developed in this note yields a substantially shorter proof of the implication ``double sphere $\Rightarrow$ projection'' established in \cite{PeraltaAOT2018}, as well as a simplified proof of the reverse implication, known under stronger hypotheses, that also applies to non-unital algebras. 

The main novelty in the argument that we present here rests on the following observation: given a non-zero projection $p$ and a norm-one positive element $b$ in a C$^*$-algebra $A$, the element $c=(\one-p)b(\one-p)$, computed in the unitisation, $\widetilde A$, of $A$, belongs to $A$ (even if $A$ is not unital). If $c\ne0$, then $d=\frac12p+\frac{c}{\|c\|}$
belongs to $S_{A^+}$ and satisfies $\|d-p\|=1$, whereas
$\|d-b\|<1$. Thus $b$ does not belong to the double sphere of $p$ unless
$c=0$. Positivity then gives $b \in pAp$, and a simple spectral argument finishes the proof (cf. Theorem~\ref{thm:cstar-projections}).

After dealing with the arguments leading to the proof of the desired conjecture in the setting of C$^*$-algebras, we subsequently explain separately why the same test elements and estimates work in an arbitrary (not necessarily unital) JB$^*$-algebra, including the exceptional case (see Theorem~\ref{thm:jb-projections}). 

In the final section, we explore a new metric characterisation of non-zero tripotents in a JB$^*$-triple $E$ using double spheres and the positive sphere of the inner ideal $E(a)$ generated by a single element $a$ in the unit sphere of $E$ and its natural structure of JB$^*$-algebra. We concretely show that a norm-one element $a$ in $E$ is a tripotent if, and only if, $\Sph_{S_{E(a)^{+}}}
\bigl(\Sph_{S_{E(a)^{+}}} (a)\bigr)=\{a\}$ (see Proposition~\ref{prop:range-positive-region}). More generally, the inner ideal $E(a)$ can be replaced by the positive region of $S_E$ associated with any tripotent $e\in E^{**}$ for which $a$ is positive in $E^{**}_2(e)$ (Theorem~\ref{thm:triple-positive-region}); for C$^*$-algebras this yields a metric characterisation of partial isometries (Corollary~\ref{cor:partial-isometries}).

\section{A metric characterisation of projections in general C\texorpdfstring{$^*$}{*}-algebras}
\label{sec:cstar}

Throughout this note, for a C$^*$-algebra $A$, we write $\widetilde A$ for its unitisation if $A$ is non-unital, and we set $\widetilde A=A$ otherwise. It follows from the definitions that $A$ is a two-sided closed ideal of $\widetilde A$. The unit of $\widetilde A$ will be denoted by $\one$. Note that we can always assume that $\one$ is also the unit
of the bidual, $A^{**}$, of $A$. 

We shall rely on two elementary facts. First, the distance between any two positive elements in the closed unit ball of a C$^*$-algebra $A$ is bounded by $1$. Namely, if $x$ and $y$ are positive
contractions in $A$, then $\|x-y\|\le1$, since $-\one\le x-y\le\one$ in $\widetilde{A}$. Moreover,
\begin{equation}\label{eq:state-endpoints}
\|x-y\|=1
\quad\Longleftrightarrow\quad
\text{there is a state $\varphi$ of $A$ with }
(\varphi(x),\varphi(y))\in\{(1,0),(0,1)\}.
\end{equation}
To see the forward implication, since $x-y$ is self-adjoint, the basic theory of C$^*$-algebras assures the existence of a state $\varphi\in A^*$ satisfying $|\varphi (x-y)| = \|x-y\|=1$ (cf. \cite[Theorem 3.3.6]{MurphyBook}), and hence $\varphi (x-y)= \pm 1$. Since $0\le\varphi(x),\varphi(y)\le 1$, the pair $(\varphi(x),\varphi(y))$ is $(1,0)$ or $(0,1)$. The converse implication follows by just applying $\varphi$ to $x-y$.

Second, for any state $\varphi$ of $A$ and any positive contraction $x$,
\begin{equation}\label{eq:square-endpoints}
\varphi(x)\in\{0,1\}
\quad\Longrightarrow\quad
\varphi(x^2)=\varphi(x).
\end{equation}
For $\varphi(x)=0$, this holds because $0\le x^2\le x$, and $\varphi$ preserves the partial order. For $\varphi(x)=1$,
the Cauchy--Schwarz inequality, applied to the extension of $\varphi$ to $\widetilde{A}$, gives
$1=\varphi(x)^2\le\varphi(x^2) \varphi(\one)\le\varphi(x)=1$, since $x^2\le x$. 

The next lemma will provide a simplified proof of \cite[Proposition 2.2]{PeraltaAOT2018}.

\begin{lemma}\label{lem:square}
Let $A$ be a C$^*$-algebra and let $a\in S_{A^+}$. Then
\[
a^2\in\Sph_{S_{A^+}}\bigl(\Sph_{S_{A^+}}(a)\bigr).
\]
\end{lemma}
\begin{proof}
We know that $a^2\ge0$ and $\|a^2\|=1$. Take
$x\in\Sph_{S_{A^+}}(a)$. By \eqref{eq:state-endpoints}, there is a state $\varphi$ of
$A$ taking the values $(1,0)$ or
$(0,1)$ on $(a,x)$. By \eqref{eq:square-endpoints}, $\varphi$ takes the same values
on $(a^2,x)$. Therefore $\|a^2-x\|=1$, again by \eqref{eq:state-endpoints}. The arbitrariness of $x$ in $\Sph_{S_{A^+}}(a)$ implies that
$a^2\in\Sph_{S_{A^+}}\bigl(\Sph_{S_{A^+}}(a)\bigr)$.
\end{proof}

The next technical lemma shows that elements in the positive sphere of $p A p$ different from $p$ cannot belong to the double sphere around $p$ in $S_{(pAp)^+}$ or in $S_{A^+}$.

\begin{lemma}\label{lem:corner}
Let $A$ be a C$^*$-algebra, let $p\in A$ be a non-zero projection,
and let $b\in S_{(pAp)^+}$ with $b\ne p$. Then there is an element
$x\in S_{(pAp)^+}$ such that
$$
\|p-x\|=1,\hbox{ and }\  \|b-x\|<1.
$$
More precisely, if $m:=\min\sigma_{pAp}(b)$, then $0\le m<1$ and
\[
x:=\frac{b-mp}{1-m}
\qquad\text{satisfies}\qquad
\|p-x\|=1, \hbox{ and } \|b-x\|=m.
\]
\end{lemma}

\begin{proof} The corner $pAp$ is a C$^*$-subalgebra of $A$ with unit $p$. Let $m:=\min\sigma_{pAp}(b)$. Since $b$ is a positive norm-one element with $b\neq p$, Gelfand theory gives $1\in \sigma_{pAp}(b)\subseteq[0,1]$ and $0\leq m\le 1$; moreover $m<1$, since $m=1$ would give $\sigma_{pAp}(b)=\{1\}$ and hence $b=p$. 

The affine function $f(t)=\frac{t-m}{1-m}$ maps the spectrum of $b$ strictly increasingly into $[0,1]$, sending the endpoints $m$ and $1$ to $0$ and $1$, respectively. Applying the continuous functional calculus within $pAp$ yields $x = f(b) \in pAp$, $x\ge0$, $\|x\|=1$, $0 \in \sigma_{pAp}(x)$ and $\|p-x\|= \max_{t\in\sigma_{pAp}(b)}\frac{1-t}{1-m} = 1$ (note that $p-x=\frac{p-b}{1-m}$). Finally,
\[
b=mp+(1-m)x,\qquad b-x=m(p-x),
\]
so $\|b-x\|=m<1$. Note that the argument also works for $m=0$, in which case $x=b$.
\end{proof}

Just one technical observation stands between us and the main results of this section.

\begin{lemma}\label{lem:detector}
Let $A$ be a C$^*$-algebra, let $p\in A$ be a non-zero projection, and let $b\in S_{A^+}$. Suppose that the element $c = b - p b - b p + p b p\in A$ is non-zero. Then $$d:=\frac12p+\frac{c}{\|c\|} \hbox{ lies in  } S_{A^+}, \hbox{ and satisfies }\|d-p\|=1 \hbox{ and } \|d-b\|<1.$$
\end{lemma}

\begin{proof}
The identity $c=b-pb-bp+pbp = (\mathbf{1}-p)b(\mathbf{1}-p)$
shows that $c\in A^+$. We also have $\|c\|\in(0,1]$, because $\|c\|\le\|b\|=1$. Since $pc=cp=0$, $p$ and $c$ are orthogonal, so
$$
\|d\|= \max\left\{\tfrac12 \|p\| , \left\| \tfrac{c}{\|c\|} \right\|\right\}=\max\bigl\{\tfrac12,1\bigr\}=1,\hbox{ and }
\|d-p\|=\left\|-\tfrac12p+\tfrac{c}{\|c\|}\right\|=\max\bigl\{\tfrac12,1\bigr\}=1.
$$
Let us finally show that $\|d-b\|<1$. Since $\frac{c}{\|c\|}$ is a positive norm-one element of the corner $(\one-p)\widetilde A(\one-p)$, whose unit is $\one-p$, we have $0\le \frac{c}{\|c\|}\le \one - p$, and thus
\begin{equation}\label{eq:detector-bounds}
\tfrac12p\le d\le\tfrac12p+(\one - p)=\one-\tfrac12p.
\end{equation} Suppose that $\|d-b\|=1$. By applying \eqref{eq:state-endpoints} we obtain a state $\varphi$ of $A$ such that 
\begin{equation}\label{eq:state-dichotomy}
(\varphi(d),\varphi(b))=(1,0)\quad\text{or}\quad(\varphi(d),\varphi(b))=(0,1).
\end{equation}
In either case, \eqref{eq:detector-bounds} implies that $\varphi(p)=0$, using the upper bound for $d$ if $\varphi(d)=1$, and the lower bound if $\varphi(d)=0$. The Cauchy--Schwarz inequality then gives
$\varphi(pb)=\varphi(bp)=0$, while
$0\le pbp\le p$ implies that $\varphi(pbp)=0$. Thus $\varphi(c)=\varphi(b)$ and
$\varphi(d)=\frac{\varphi(b)}{\|c\|}$.
This contradicts both alternatives in \eqref{eq:state-dichotomy}. Since $d$ and $b$ are positive contractions, their distance is at most one, and we must necessarily have $\|d-b\|<1$.
\end{proof}

\begin{remark} In Lemma~\ref{lem:detector}, the coefficient $\frac12$ in $d$ could be replaced by any $\lambda\in(0,1)$. Namely, by considering the inequalities $\lambda p\le d\le\one-(1-\lambda) p$, the same
argument remains valid. In particular, any state $\varphi$ taking the
value $0$ or $1$ at $d$ must vanish at $p$.
\end{remark}

In our main result, we not only prove the desired conjecture for not necessarily unital C$^*$-algebras, but we also give a simplified proof of the implication ``double sphere $\Rightarrow$ projection'' originally established in \cite[Proposition 2.2]{PeraltaAOT2018}.

\begin{theorem}\label{thm:cstar-projections} Let $A$ be a C$^*$-algebra, and let
$a$ be a norm-one positive element in $A$. Then the following statements are equivalent:
\begin{enumerate}[$(i)$]
	\item $a$ is a non-zero projection;
	\item $\Sph_{S_{A^+}}\bigl(\Sph_{S_{A^+}}(a)\bigr)=\{a\}$.
\end{enumerate}
\end{theorem}

\begin{proof}
If $(ii)$ holds, Lemma~\ref{lem:square} yields $a^2\in\Sph_{S_{A^+}}\bigl(\Sph_{S_{A^+}}(a)\bigr)=\{a\}$, that is, $a^2=a$. Since $a$ is positive, Gelfand theory then implies that $a$ is a non-zero projection. 

Conversely, let $p\in A$ be a non-zero projection and take
$b\in\Sph_{S_{A^+}}\bigl(\Sph_{S_{A^+}}(p)\bigr)$.
If the element $c = b - p b - b p + p b p \in A$ is non-zero, Lemma~\ref{lem:detector} provides an element
$d\in\Sph_{S_{A^+}}(p)$ with $\|d-b\|<1$, which contradicts that $b\in\Sph_{S_{A^+}}\bigl(\Sph_{S_{A^+}}(p)\bigr)$. Consequently $c=0$. 

Working in the unitization of $A$ if necessary, we obtain $c=(\one-p)b(\one-p)=0.$ For convenience, we henceforth write $p^\perp=\one-p$. Since $b$ is a positive element, the equality $0=p^\perp b p^\perp=(b^{1/2}p^\perp)^*(b^{1/2}p^\perp)$
implies $b^{1/2}p^\perp=0$. Hence $bp^\perp=p^\perp b=0$, so $b=pbp\in pAp$.
If $b\ne p$, Lemma~\ref{lem:corner} assures the existence of an element
$x\in S_{(pAp)^+}\subseteq S_{A^+}$ with $\|p-x\|=1$ and
$\|b-x\|<1$, which again contradicts the choice of $b$. Thus $b=p$ and $\{p\} = \Sph_{S_{A^+}}\bigl(\Sph_{S_{A^+}}(p)\bigr)$.
\end{proof}

\begin{remark}\label{rem:scalar}
Note that the convention $\Sph_P(\emptyset)=P$ is needed in the statement of Theorem~\ref{thm:cstar-projections}. For
instance, if $A=\C$, then $S_{A^+}=\{\one\}$ and the sphere
$\Sph_{S_{A^+}}(\one)$ is empty, so the double sphere around $\one$ in 
$S_{A^+}$ is 
$S_{A^+} = \{\one\}$, as the theorem requires.
\end{remark}

\section{A metric characterisation of projections in general JB\texorpdfstring{$^*$}{*}-algebras}
\label{sec:jb}

We begin by observing that every C$^*$-algebra $A$ can be regarded as a (not necessarily associative) JB$^*$-algebra keeping the original norm and involution, and replacing the associative product by the Jordan product given by $a\circ b = \frac12 (a b +b a)$ ($a,b\in A$). Formally speaking, a \emph{JB$^*$-algebra} is a complex Banach space $\mathfrak{A}$ together with a bilinear and commutative product ($(a,b)\mapsto a\circ b$) satisfying $\|a\circ b\|\le\|a\|\,\|b\|$, and an algebra involution ($a\mapsto a^*$), such that the following axioms hold: 
\begin{enumerate}[$(1)$]
	\item $(a\circ b)\circ b^2  = (a\circ b^2)\circ b$\ \ \ \ ($\forall a,b\in \mathfrak{A}$) \hfill (\emph{Jordan identity}),
	\item If for each $a\in \mathfrak{A}$ the symbol $U_{a}$ stands for the quadratic operator given by $U_a (b) = 2 (a\circ b)\circ a - a^2 \circ b$  ($\forall b\in \mathfrak{A}$), we have $\left\| U_{a} (a^*) \right\| = \|a\|^3$ ($\forall a\in \mathfrak{A}$)  \hfill (\emph{Gelfand--Naimark axiom}). 
\end{enumerate} It should be noted that the class of JB$^*$-algebras is strictly larger than the class of C$^*$-algebras. For the standard definitions and properties of JB$^*$-algebras employed in this paper, we refer to \cite{AlfsenShultz2003,HancheOlsenStormer,Cabrera-Rodriguez-vol1} and to the detailed exposition in \cite[\S 2]{PeSa2025} and \cite[\S 3]{PeraltaSaavedraUnital}. 

Let $\mathfrak{A}$ be a JB$^*$-algebra. We say that $\mathfrak{A}$ is unital if there exists $\one\in \mathfrak{A}$ satisfying $\one \circ a =a$ for all $a\in \mathfrak{A}$. A projection in $\mathfrak{A}$ is a self-adjoint idempotent. For a projection
$p\in \mathfrak{A}$, the subalgebra $U_p(\mathfrak{A})$ is a unital JB$^*$-subalgebra with unit $p$. An element $a$ in $\mathfrak{A}$ is called positive ($a\geq 0$ in short) if $a =a^*$ and its Jordan spectrum is contained in $\mathbb{R}_0^+$; equivalently, $a$ is the square of an element in $\mathfrak{A}_{\mathrm{sa}}:=\{a\in \mathfrak{A}: a^* =a \}$ (see \cite[3.2.4 and 3.3.3]{HancheOlsenStormer}). We shall write $\mathfrak{A}^+$ for the set of
positive elements in $\mathfrak{A}$. As in the associative setting, given a JB$^*$-algebra $\mathfrak{A}$ we shall write $S_{\mathfrak{A}^+}$ for its positive unit sphere, that is, the set of all norm-one positive elements in $\mathfrak{A}$. 

The second dual, $\mathfrak{A}^{**},$ of each JB$^*$-algebra $\mathfrak{A}$, admits a structure of unital JB$^*$-algebra in such a way that $\mathfrak{A}$ embeds isometrically as a weak$^*$-dense JB$^*$-subalgebra inside it. If $\mathfrak{A}$ is not unital, the symbol $\one$ will stand for the unit in $\mathfrak{A}^{**}$.

Although there are examples of exceptional JB$^*$-algebras which cannot be embedded as JB$^*$-subalgebras of any C$^*$-algebra, the celebrated Shirshov--Cohn theorem (see \cite[Theorems 2.4.14 and 7.2.5]{HancheOlsenStormer} or \cite[Section~2.2]{LemmensRoelandsWortel}) assures that the JB$^*$-subalgebra generated by any two self-adjoint elements (and the unit element) in an arbitrary JB$^*$-algebra is always a JB$^*$-subalgebra of some C$^*$-algebra, and the C$^*$-algebra can be chosen with the same unit. Since this representation is only local, it does not transfer the double-sphere condition; we use it solely to verify norm estimates for finitely many explicit elements.

\begin{lemma}\label{lem:jb-detector}
Let $\mathfrak{A}$ be a JB$^*$-algebra, let $p\in \mathfrak{A}$ be a non-zero projection, and let $b\in S_{\mathfrak{A}^+}$. Then, the element
\begin{equation}\label{eq:Jordan-c-definition}
c:=U_{\one-p}(b) =b-2p\circ b+U_p(b)\in \mathfrak{A}
\end{equation}
is positive in $\mathfrak{A}$ (note that $c\in\mathfrak{A}$ by the last expression, even if $\one\notin\mathfrak{A}$), and if $c\ne 0$, by setting $
d:=\frac12p+\frac{c}{\|c\|},$ we get another element in $S_{\mathfrak{A}^+}$ satisfying $\|d-p\|=1$ and $\|d-b\|<1$.
If $c=0$, then $b=U_p(b)$ and $0\le b\le p$.
\end{lemma}

\begin{proof} Let $\mathfrak{B}$ denote the JB$^*$-subalgebra of $\mathfrak{A}^{**}$ generated by $p,b$ and $\one$. By the Shirshov--Cohn theorem (see \cite[Theorem 7.2.5]{HancheOlsenStormer}), there exists a unital C$^*$-algebra $A$ containing $\mathfrak{B}$ as JB$^*$-subalgebra and admitting $\one$ as the unit element. The associative C$^*$-product on $A$ will be denoted by juxtaposition. 
	
Note that $b\in S_{\mathfrak{A}^+}\cap \mathfrak{B} = S_{\mathfrak{B}^+}\subseteq S_{A^+}$.  Lemma~\ref{lem:detector}, applied to $A$, assures that the element $c=(\one-p)b(\one-p) = U_{\one-p} (b) =b-2p\circ b+U_p(b) \geq 0$ in $A$ (equivalently, in $\mathfrak{B}$ and in $\mathfrak{A}$ because $c$ trivially lies in $\mathfrak{B}$). Furthermore, by the same lemma, if $c\neq 0$, the element $d:=\frac12p+\frac{c}{\|c\|}$ lies in $S_{A^+}$ and satisfies $\|d-p\|_{A}=1$ and $\|d-b\|_{A}<1$. Having in mind that $p,b,c,$ and $d$ all lie in $\mathfrak{B}\cap \mathfrak{A}$, it trivially holds that $d\in S_{\mathfrak{A}^+}$ with $\|d-p\|=\|d-p\|_{\mathfrak{B}}=\|d-p\|_{A}=1$ and $\|d-b\|=\|d-b\|_{\mathfrak{B}}=\|d-b\|_{A}<1.$  Clearly, if $c=0$, then $b^{1/2}(\one-p)=0$ in $A$ (as in the proof of Theorem~\ref{thm:cstar-projections}), and thus $b= p b p = U_p (b) \leq U_p (\one) = p$.
\end{proof}

\begin{theorem}\label{thm:jb-projections}
Let $\mathfrak{A}$ be a JB$^*$-algebra, and let
$a\in S_{\mathfrak{A}^+}$. The following statements are equivalent: \begin{enumerate}[$(i)$] \item $a$ is a non-zero projection;
\item $\Sph_{S_{\mathfrak{A}^+}}\bigl(\Sph_{S_{\mathfrak{A}^+}}(a)\bigr)=\{a\}$.
\end{enumerate}
\end{theorem}

\begin{proof} We shall first check that 
\begin{equation}\label{eq lemma square Jordan} 
	b^2\in\Sph_{S_{\mathfrak{A}^+}}\bigl(\Sph_{S_{\mathfrak{A}^+}}(b)\bigr)
	\qquad(\forall b\in S_{\mathfrak{A}^+}).
\end{equation}
Indeed, for $x\in\Sph_{S_{\mathfrak{A}^+}}(b)$, let $\mathfrak{B}$ denote the JB$^*$-subalgebra of $\mathfrak{A}$ generated by $x$ and $b$. By the Shirshov--Cohn theorem, there exists a unital C$^*$-algebra $A$ containing $\mathfrak{B}$ as JB$^*$-subalgebra. Note that $x,b,b^2\in S_{\mathfrak{B}^+}\subseteq S_{A^+}$ with $\|b-x\|=1.$ Therefore $x\in  \Sph_{S_{{A}^+}}(b)$. Now, Lemma~\ref{lem:square} assures that $1=\|b^2-x\|_{A} = \|b^2-x\|_{\mathfrak{B}} = \|b^2-x\|_{\mathfrak{A}}$. The arbitrariness of $x\in\Sph_{S_{\mathfrak{A}^+}}(b)$ concludes the proof of \eqref{eq lemma square Jordan}.

$(ii)\Rightarrow (i)$. It follows from \eqref{eq lemma square Jordan} that $a^2\in \Sph_{S_{\mathfrak{A}^+}}\bigl(\Sph_{S_{\mathfrak{A}^+}}(a)\bigr)=\{a\}$, so $a^2 = a$, and $a$ is a non-zero projection since $\|a\|=1$.

$(i)\Rightarrow (ii)$. Let $p\in \mathfrak{A}$ be a non-zero projection and let $b\in\Sph_{S_{\mathfrak{A}^+}}\bigl(\Sph_{S_{\mathfrak{A}^+}}(p)\bigr)$.
Lemma~\ref{lem:jb-detector} forces $U_{\mathbf{1}-p}(b)=0$, as otherwise we could construct an element $d \in \Sph_{S_{\mathfrak{A}^+}}(p)$ with $\|d-b\|<1$, which is impossible. Therefore, we must have $b=U_p(b)$ and $0\le b\le p$.

It remains to work in the unital JB$^*$-algebra $U_p(\mathfrak{A})$, whose unit is $p$. Let $C$ be the JB$^*$-subalgebra of $U_p(\mathfrak{A})$ generated by $b$ and $p$. It is known that $C$ is an associative JB$^*$-algebra \cite[3.2.3 and 3.2.4]{HancheOlsenStormer}, and thus a commutative unital C$^*$-algebra under its own product.

If $b\ne p$, define
$m := \min\sigma_{U_p(\mathfrak{A})}(b) = \min \sigma_{C}(b)$. Continuous functional calculus in $C$ (cf. \cite[3.2.4]{HancheOlsenStormer}) or a direct application of Lemma~\ref{lem:corner} to $p,b\in C$ gives $$
0\le m<1,\qquad
x:=\frac{b-mp}{1-m}\in S_{C^+} \subseteq S_{\mathfrak{A}^+},\qquad \|p-x\|=1,\qquad \hbox{and}\qquad \|b-x\|=m<1,$$
which contradicts the choice of $b\in\Sph_{S_{\mathfrak{A}^+}}\bigl(\Sph_{S_{\mathfrak{A}^+}}(p)\bigr)$. Therefore $b=p,$ and consequently  $\Sph_{S_{\mathfrak{A}^+}}\bigl(\Sph_{S_{\mathfrak{A}^+}}(p)\bigr)\subseteq \{p\}$, while the reverse inclusion is clear.
\end{proof}

\section{Local cones and tripotents in JB\texorpdfstring{$^*$}{*}-triples}
\label{sec:triples}

C$^*$- and JB$^*$-algebras belong to a strictly wider class of complex Banach spaces known as JB$^*$-triples. For a C$^*$-algebra, the triple product is given by
\begin{equation}\label{eq C*-triple product} \{a,b,c\}=\frac{1}{2}(ab^*c+cb^*a),
\end{equation}
whereas for a JB$^*$-algebra, it is given by $$\{a,b,c\}=(a\circ b^*)\circ c
+(c\circ b^*)\circ a
-(a\circ c)\circ b^*.$$ More generally, a complex Banach space $E$ is a JB$^*$-triple if and only if its open unit ball is a bounded symmetric domain \cite{Kaup83riemann}. In this case, the triple product \((a,b,c)\mapsto\{a,b,c\}\) is uniquely determined by the Banach space structure \cite[Proposition 5.5]{Kaup83riemann}, and satisfies the following axioms:\begin{enumerate}[(T1)]\item The triple product is symmetric and bilinear in the outer variables and conjugate linear in the middle,
\item $L(x,y)  \{u,v,w\} = \{{L(x,y)u},v,w\} - \{ u,{L(y,x)v},w\} + \{u,v,{L(x,y)w}\},$ where $L(x,y) (z) = \{x,y,z\}$ ($x,y,z\in E$),
\item $\|L(x,x)\| = \|x\|^2$ and $L(x,x)$ is a hermitian operator on $E$ with non-negative spectrum. 
\end{enumerate} 

Other examples of JB$^*$-triples include the Cartan factors $B(H,K)$ of bounded linear operators between complex Hilbert spaces $H$ and $K$, spin factors, and the exceptional Cartan factors $H_3(\mathbb{O})$ and $M_{1,2}(\mathbb{O})$, consisting of hermitian \(3\times3\) matrices and \(1\times2\) matrices, respectively, over the complex octonions \(\mathbb{O}\) (see \cite[\S 2.2.27 and 4.1.3]{Cabrera-Rodriguez-vol1} for more details and background).

An element $e$ in a JB$^*$-triple $E$ is called a tripotent if $\{e,e,e\}=e$. If a C$^*$-algebra $A$ is regarded as a JB$^*$-triple with respect to the triple product in \eqref{eq C*-triple product}, the tripotents in $A$ are precisely the partial isometries. Every tripotent $e\in E$ induces a Peirce decomposition $$
E=E_2(e)\oplus E_1(e) \oplus E_0(e),$$
where $E_j(e)$ denotes the $\frac{j}{2}$-eigenspace of the operator $
L(e,e)(x)=\{e,e,x\}$. The Peirce-$2$ subspace $E_2(e)$ is a JB$^*$-algebra with product $x\circ_e y=\{x,e,y\}$ and involution $x^{*_e}=\{e,x,e\}$ (cf. \cite[Corollary 4.2.30]{Cabrera-Rodriguez-vol1}). For a tripotent $e$ in a JB$^*$-triple $E$, writing $Q(e)(x)=\{e,x,e\}$, the Peirce projections of $E$ onto the corresponding Peirce subspaces are given by the following identities:
\begin{enumerate}[$\bullet$]
\item $P_2(e) = Q(e)^2 = L(e,e)(2L(e,e) - \mathrm{Id}_E),$
\item $P_1(e) = 4L(e,e)(\mathrm{Id}_E - L(e,e)),$ 
\item $P_0(e) = (\mathrm{Id}_E - L(e,e))(\mathrm{Id}_E - 2L(e,e)) = \mathrm{Id}_{E} - 2 L(e,e) + Q(e)^2.$ 
\end{enumerate} Furthermore, the Peirce spaces $E_j(e)$ ($j \in \{0,1,2\}$) satisfy the following multiplication \emph{Peirce rules} under the triple product: $$\{E_i(e), E_j(e), E_k(e)\} \subseteq E_{i-j+k}(e)$$ where $E_{i-j+k}(e) = \{0\}$ if $i-j+k \notin \{0, 1, 2\}$.

Two tripotents $u, e$ in a JB$^*$-triple $E$ are \emph{orthogonal} ($u \perp e$ in short) if $L(u, e) = 0$ (or equivalently, $u \in E_0(e)$). The natural partial order among tripotents in $E$ is defined by:$$u \le e \iff e - u \text{ is a tripotent orthogonal to } u.$$ Equivalently, $u \le e$ holds if, and only if, $u = P_2(u)(e)$ (or $u = \{u, u, e\}$) if, and only if, $u$ is a projection in $E_2(e)$.

A JBW$^*$-triple is a JB$^*$-triple that is a dual Banach space. Such a space has a unique isometric predual, and its triple product is separately weak$^*$-continuous \cite{BartonTimoney}. Moreover, the bidual $E^{**}$ of any JB$^*$-triple $E$ is itself a JBW$^*$-triple, in which $E$ embeds canonically as a JB$^*$-subtriple \cite{dineen86complete}. The \emph{range tripotent}, $\ran(a)$, of an element $a$ in a JBW$^*$-triple $W$ is the
smallest tripotent $e\in W$ such that $a$ is positive in the JBW$^*$-algebra $W_{2} (e)$. 

A closed subspace $I$ of a JB$^*$-triple $E$ is called an \emph{inner ideal} of $E$ if $\{I,E, I\}\subseteq I$. In what follows we gather several structural properties of the inner ideal, $E(a)$, generated by an element $a$ in a JB$^*$-triple $E$ taken from \cite[Proposition 2.1]{BuCHuZa2000MathScand}. The inner ideal $E(a)$ coincides with the norm closure of the image of $Q(a)$, where $Q(a)(x)=\{a,x,a\}$, and satisfies the following properties:
\begin{enumerate}[$(1)$]
	\item The JB$^*$-subtriple $E_a$ generated by $a$ inside $E$, which is isometrically isomorphic to a commutative C$^*$-algebra in which $a$ is a positive generator, is contained in $E(a)$.
	\item If $\ran(a)$ denotes the range tripotent of $a$ in $E^{**}$, the inner ideal $E(a)$ is a JB$^*$-subalgebra of the Peirce-$2$ space $E^{**}_2 (\ran(a))$.
\end{enumerate}

We state next a first metric characterisation of non-zero tripotents in a JB$^*$-triple.

\begin{proposition}\label{prop:range-positive-region} Let $a$ be a norm-one element in a JB$^*$-triple $E$, and let $E(a)$ denote the inner ideal of $E$ generated by $a$. Then the following statements are equivalent: \begin{enumerate}[$(i)$] \item $a$ is a tripotent in $E$,
\item $a$ is the unit element in $E(a)$,
\item $\Sph_{S_{E(a)^+}}
		\bigl(\Sph_{S_{E(a)^+}}(a)\bigr)=\{a\}$.
\end{enumerate}
\end{proposition}

\begin{proof} Let $r=\ran(a)$ denote the range tripotent of $a$ in $E^{**}$. Having in mind that $E(a)$ is a JB$^*$-subalgebra of $(E_2^{**}(r),\circ_{r},*_{r})$, and $a$ is a positive element in $E(a)$, it follows from Theorem~\ref{thm:jb-projections} that $a$ is a projection in $E(a)$ if, and only if, $\Sph_{S_{E(a)^+}}
	\bigl(\Sph_{S_{E(a)^+}}(a)\bigr)=\{a\}$. However, by the uniqueness of the triple product in any JB$^*$-triple (see  \cite[Proposition 5.5]{Kaup83riemann}), we have $$\{x,y,z\} =  (x\circ_{r} y^{*_{r}})\circ_{r} z + (z\circ_{r} y^{*_{r}})\circ_{r} x - (x\circ_r z)\circ_r y^{*_{r}}$$ for all $x,y,z\in E(a)$. Therefore $a$ is a tripotent in $E$ if, and only if, $a= \{a,a,a\} = (a\circ_r a)\circ_r a,$ that is, $a^3=a$ in $E(a)$, which, for a positive element, is equivalent to saying that $a$ is a projection in $E(a)$. This proves $(i)\Longleftrightarrow(iii)$. Finally, if $a$ is the unit of $E(a)$, then it is a projection there, so $(ii)\Rightarrow(i)$; conversely, if $a$ is a tripotent, then $Q(a)^2=P_2(a)$ shows that $E(a)=E_2(a)$, whose unit is $a$, so $(i)\Rightarrow(ii)$.
\end{proof}

In the next result we obtain a strengthened version of the previous proposition. For a fixed tripotent $e\in E^{**}$, define the \emph{positive region of $S_E$ associated with $e$} as the set 
\begin{equation}\label{eq:positive-region}
P_e:=S_E\cap\bigl(E^{**}_2(e)\bigr)^+.
\end{equation} Note that $P_e$ could be empty. If $P_e\neq \emptyset$, the elements in $P_e$ belong to $S_{E}$, and they are positive in the JB$^*$-algebra $E^{**}_2(e)$. Clearly, the set $P_e$ has diameter at most one, and when $P_e\neq \emptyset$, we can apply the sphere construction \eqref{eq:unit-sphere-around-a-set} with $P=P_e$. We can now prove that the conclusion of Proposition~\ref{prop:range-positive-region} holds when the inner ideal generated by a norm-one element is replaced by a more general structure. 

\begin{samepage}
\begin{theorem}\label{thm:triple-positive-region}
Let $E$ be a JB$^*$-triple, let $e\in E^{**}$ be a tripotent, and
let $a\in P_e$. Then the following statements are equivalent:
\begin{enumerate}[$(i)$]
\item $a$ is a tripotent in $E$;
\item $\Sph_{P_e}\bigl(\Sph_{P_e}(a)\bigr)=\{a\}$.
\end{enumerate}
\end{theorem}
\end{samepage}

\begin{proof} Let $\mathfrak{A}$ denote the unital JB$^*$-algebra $E^{**}_2(e)$.  Take any $b_1,b_2,b_3\in E\cap E^{**}_2(e)$ (a subspace which contains $a$ by hypothesis). It follows from Peirce rules and the fact that $E$ is a subtriple of $E^{**}$ that $\{b_1,b_2,b_3\}\in E\cap E^{**}_2(e)$. Set $a^{[3]} = \{a,a,a\}$. Since $a\in P_e$ and $\|a^{[3]}\| = \|a\|^3 =1$, we deduce that $a^{[3]}\in S_{E}\cap E^{**}_2 (e)$. Furthermore, the triple product on each JB$^*$-triple is unique (cf. \cite[Proposition 5.5]{Kaup83riemann}) and hence on the JB$^*$-algebra $E_2^{**} (e)$ we have $$a^{[3]} = \{a,a,a\} = 2 a \circ_e (a^{*_e} \circ_e a)- a^{*_e} \circ_e (a \circ_e a) = a \circ_e (a \circ_e a) \in \left(E^{**}_2 (e)\right)^+,$$ because $a\in \left(E^{**}_2 (e)\right)^+$. Therefore $a^{[3]}\in P_e$.
	
Now, take $x\in\Sph_{P_e}(a)=\left\{z\in P_e : \|z-a\|=1\right\}$. By the Shirshov--Cohn theorem we can find a unital C$^*$-algebra $A$ containing the JB$^*$-subalgebra of $E^{**}_2 (e)$ generated by $a,x$ and $e$, with the additional property that $e$ is also the unit of $A$. Clearly $a$, $x$ are elements in $S_{A^+}$ with $\|a-x\|=1$. By \eqref{eq:state-endpoints}, there is a state
 $\varphi$ of $A$ satisfying
 \begin{equation}\label{eq:triple-state-dichotomy}
 	(\varphi(a),\varphi(x))=(1,0)\quad\text{or}\quad(\varphi(a),\varphi(x))=(0,1).
 \end{equation}
 Since $t^3 \leq t$ and $1-t^3=(1-t)(1+t+t^2)\leq 3(1-t)$ for all $t \in [0,1]$, by applying the continuous functional calculus in $A$ for $a$, we get
 \[
 0\le a^3\le a,\qquad 0\le e-a^3\le3(e-a).
 \]
 Hence, $\varphi(a^3)=\varphi(a)$ in both cases of \eqref{eq:triple-state-dichotomy}, and a new application of \eqref{eq:state-endpoints} gives
 $\|a^{[3]}-x\|_{E}=\|a^{[3]}-x\|_{E^{**}_2 (e)}=\|a^{[3]}-x\|_{A}=\|a^3-x\|_{A}=1$. Thus
 \[
 a^{[3]} \in \Sph_{P_e}\bigl(\Sph_{P_e}(a)\bigr).
 \]
Consequently, if $(ii)$ holds, we deduce that $\{a,a,a\}=a$, and hence $a$ is a tripotent (and, clearly, $a \leq e$).  

Let us finally prove that $(i)\Rightarrow (ii)$. Suppose $a = v\in P_e$ is a tripotent in $E$, and take $b\in\Sph_{P_e}\bigl(\Sph_{P_e}(v)\bigr)$. It follows from the assumptions that $v$ is a projection in the unital JB$^*$-algebra $E^{**}_2 (e)$, and moreover $E_2 (v)$ is a JB$^*$-subalgebra of $E^{**}_2 (e)$. Set $c=P_0(v)(b)\in E$. By a new application of the uniqueness of the triple product of $E_2^{**} (e)$ (cf. \cite[Proposition 5.5]{Kaup83riemann}), and having in mind that $b$ is positive in $E^{**}_2(e)$, we arrive at  \begin{equation}\label{eq:triple-complementary-compression}\begin{aligned}c&=P_0(v)(b) = b  - 2 L(v,v) (b) + Q(v)^2 (b) = b  - 2 \{v,v,b\} + \{v,\{v,b,v\},v\} \\
&= b - 2 (v\circ_e v^{*_e})\circ_e b -2 (b\circ_e v^{*_e})\circ_e v + 2 (v\circ_e b)\circ_e v^{*_e}+ \{v, U_{v} (b^{*_e}) , v \}\\
&= b - 2 v\circ_e b + U_v^2 (b) = b - 2 v\circ_e b + U_{v\circ_e v} (b) = b - 2 v\circ_e b + U_{v} (b) 
= U_{e-v} (b).
	\end{aligned} 
\end{equation}

If $c\ne 0$, Lemma~\ref{lem:jb-detector}, applied to the elements $v$ and $b$ in the JB$^*$-algebra $E^{**}_2(e)$, implies that  $d=\frac12 v+\frac{c}{\|c\|}\in E$ is positive in $E_2^{**}(e)$, $
\|d\|=\| d - v \|=1,$ and $\|d-b\|<1.$
Therefore, $d\in P_e$ and hence $d\in\Sph_{P_e}(v)$ (since $\| d - v \|=1$). However, the inequality  $\|d-b\|<1$ contradicts the assumption that $b\in\Sph_{P_e}\bigl(\Sph_{P_e}(v)\bigr)$. Therefore, $c=0$, and the
same Lemma~\ref{lem:jb-detector} yields $b=U_v (b)=P_2(v)(b)$ and $0\le b\le v$, which in particular assures that $b\in  S_{\left(E_2(v)\right)^+}$.

Obviously $\Sph_{S_{\left(E_2(v)\right)^+}}(v) = \Sph_{P_e}(v) \cap S_{\left(E_2(v)\right)^+}\subseteq \Sph_{P_e}(v)$, and thus by Theorem~\ref{thm:jb-projections} (or by \cite[Lemma 3.1]{PeraltaSaavedraUnital}) we have $$\Sph_{S_{\left(E_2(v)\right)^+}}\left(\Sph_{P_e}(v) \right) \subseteq  \Sph_{S_{\left(E_2(v)\right)^+}}\big( \Sph_{S_{\left(E_2(v)\right)^+}}(v) \big)=\{v\}.$$ Since $b\in\Sph_{P_e}\bigl(\Sph_{P_e}(v)\bigr) \cap  S_{\left(E_2(v)\right)^+}$, we necessarily conclude that $b = v$. We have thus proved that $\Sph_{P_e}\bigl(\Sph_{P_e}(v)\bigr) =\{v\}$.
\end{proof}

In Theorem~\ref{thm:triple-positive-region} the hypothesis concerning the existence of a tripotent $e\in E^{**}$ such that $a\in P_e$ is quite natural, since we can always consider the range tripotent of $a$ in $E^{**}$.  

Note that if $a$ is a tripotent in a JB$^*$-triple $E$, then its range tripotent in $E^{**}$, $\ran(a),$ coincides with $a$, and $P_{\ran(a)}$ is the
set of positive norm-one elements of the unital JB$^*$-algebra $E_2(a)$, so in
that case Theorem~\ref{thm:triple-positive-region} reduces to the unital JB$^*$-algebra case. Its interest lies in
the converse direction: the region $P_{\ran(a)}$ is determined by $a$ itself,
without assuming in advance that $a$ is a tripotent. 

Specialised to C$^*$-algebras, Theorem~\ref{thm:triple-positive-region} yields the following characterisation of partial isometries, which seems to be new. Here, $P_e=S_A\cap\bigl(A^{**}_2(e)\bigr)^+$, where $A^{**}_2(e)=ee^*A^{**}e^*e$ is equipped with the product $x\circ_e y=\frac12(xe^*y+ye^*x)$ and the involution $x^{*_e}=ex^*e$.

\begin{corollary}\label{cor:partial-isometries}
Let $A$ be a C$^*$-algebra, let $e\in A^{**}$ be a partial isometry, and
let $a\in P_e$. The following statements are equivalent:
\begin{enumerate}[$(i)$]
	\item $a$ is a partial isometry in $A$;
	\item $\Sph_{P_e}\bigl(\Sph_{P_e}(a)\bigr)=\{a\}$.
\end{enumerate}
\end{corollary}

\begin{proof}
Apply Theorem~\ref{thm:triple-positive-region} to $E=A$ with the triple product \eqref{eq C*-triple product}, whose tripotents are the partial isometries of $A$.
\end{proof}

\section*{Acknowledgements}

\begingroup
\setlength{\emergencystretch}{3em}
\noindent A.M. Peralta is supported by grant PID2021-122126NB-C31 funded by
MICIU\slash AEI\slash 10.13039\slash 501100011033 (Ministerio de Ciencia,
Innovaci{\'o}n y Universidades y Agencia Estatal de Investigaci{\'o}n) and
ERDF/EU (European Regional Development Fund), Junta de Andaluc\'ia grant
FQM375, and Bureau of Foreign Experts Affairs, MOHRSS, PR China grant S20250924.
\par
\endgroup

\noindent P. Saavedra is supported by a ``Formaci\'{o}n de Profesorado
Universitario'' (FPU) grant from the Spanish Ministry of Science,
Innovation, and Universities (Grant No. FPU23/00747). 

\noindent A.M. Peralta and P. Saavedra are supported by grant PID2025-167660NB-I00 funded by MICIU/AEI/10.13039/501100011033 and ERDF/EU.

\noindent\textbf{Data Availability} Data sharing is not applicable to this
article as no datasets were generated or analysed during the preparation
of the paper.  

\noindent\textbf{Declarations}

\noindent\textbf{Conflict of interest} The authors declare that they have
no conflict of interest.

\raggedbottom


\begin{thebibliography}{99}
\bibitem{AlfsenShultz2003}
E.M. Alfsen, F.W. Shultz,
\emph{Geometry of state spaces of operator algebras},
Mathematics: Theory \& Applications, Birkh\"{a}user, Boston, MA, 2003.

\bibitem{BartonTimoney}
T.J. Barton, R.M. Timoney,
\emph{Weak$^*$-continuity of Jordan triple products and its applications},
Math. Scand. \textbf{59} (1986), 177--191.

\bibitem{BuCHuZa2000MathScand}
L.J. Bunce, Ch.-H. Chu, B. Zalar,
\emph{Structure spaces and decomposition in JB$^*$-triples},
Math. Scand. \textbf{86} (2000), 17--35.

\bibitem{Cabrera-Rodriguez-vol1}
M. Cabrera~Garc\'{\i}a, A. Rodr\'{\i}guez~Palacios,
\emph{Non-associative normed algebras. Vol.~1. The Vidav--Palmer and Gelfand--Naimark theorems},
Encyclopedia of Mathematics and its Applications, vol.~154, Cambridge University Press, Cambridge, 2014.

\bibitem{dineen86complete}
S. Dineen,
\emph{Complete holomorphic vector fields on the second dual of a Banach space},
Math. Scand. \textbf{59} (1986), 131--142.


\bibitem{HancheOlsenStormer}
H. Hanche-Olsen, E. St\o rmer,
\emph{Jordan operator algebras},
Monographs and Studies in Mathematics, vol.~21, Pitman, Boston, MA, 1984.


\bibitem{Kaup83riemann}
W. Kaup,
\emph{A Riemann mapping theorem for bounded symmetric domains in complex Banach spaces},
Math. Z. \textbf{183} (1983), 503--529.

\bibitem{LemmensRoelandsWortel}
B. Lemmens, M. Roelands, M. Wortel,
\emph{Hilbert and Thompson isometries on cones in JB-algebras},
Math. Z. \textbf{292} (2019), 1511--1547.

\bibitem{LeungNgWong2025}
C.-W. Leung, C.K. Ng, N.C. Wong,
\emph{Tingley's problem for positive unit spheres of operator algebras and diametral relations},
Trans. Amer. Math. Soc. \textbf{378} (2025), 2695--2711.

\bibitem{LeungNgWong2026}
C.-W. Leung, C.K. Ng, N.C. Wong,
\emph{Reconstructing the Jordan structure of C$^*$-algebras from the metric on positive unit spheres},
preprint, arXiv:2610.32029 (2026).

\bibitem{MurphyBook}
G.J. Murphy,
\emph{C$^*$-algebras and operator theory},
Academic Press, Boston, MA, 1990.

\bibitem{Nagy2018}
G. Nagy,
\emph{Isometries of spaces of normalized positive operators under the operator norm},
Publ. Math. Debrecen \textbf{92} (2018), 243--254.

\bibitem{PeraltaAOT2018}
A.M. Peralta,
\emph{Characterizing projections among positive operators in the unit sphere},
Adv. Oper. Theory \textbf{3} (2018), 731--744.

\bibitem{PeraltaSaavedraUnital}
A.M. Peralta, P. Saavedra,
\emph{A metric characterization of projections among positive norm-one elements in unital C$^*$-algebras},
Results Math. \textbf{81} (2026), Paper No.~142.

\bibitem{PeSa2025}
A.M. Peralta, P. Saavedra,
\emph{Identifying JBW$^*$-algebras through their spheres of positive elements},
Sci. China Math., to appear, doi:10.1007/s11425-025-2641-3, arXiv:2505.03287.


\end{thebibliography}
\end{document}